\documentclass[11pt,leqno]{amsart}
\usepackage{amsmath,amsthm,amssymb,graphicx} 

\theoremstyle{plain}
\newtheorem{thm}{Theorem}[section]
\newtheorem{prop}[thm]{Proposition}
\newtheorem{cor}[thm]{Corollary}
\newtheorem{lem}[thm]{Lemma}

\theoremstyle{remark}
\newtheorem{rem}[thm]{Remark}

\theoremstyle{definition}

\newtheorem{defn}[thm]{Definition}

\def\T{\mathbb{T}}
\def\I{\mathbb{I}}

\def\ds{\displaystyle}
\def\ts{\textstyle}

\DeclareMathOperator*{\esssup}{ess\,sup}
\makeatletter
\DeclareMathOperator*{\essinf}{ess\,inf}
\makeatletter
\DeclareMathOperator*{\Var}{Var}
\makeatletter

\makeatletter

\numberwithin{equation}{section}
\numberwithin{figure}{section}
\numberwithin{table}{section}

\begin{document}

\renewcommand{\thefootnote}{\fnsymbol{footnote}}








\title{Dynamic risk diversification and insurance premium principles}
\author{
Kei Fukuda, Akihiko Inoue and 
Yumiharu Nakano}
\address{Japan Credit Rating Agency, Ltd. \\
Jiji Press Building, 5-15-8 Ginza, Chuo-Ku \\
Tokyo 104-0061, Japan}
\email{fukuda@jcra.com}
\address{Department of Mathematics \\
Graduate School of Science \\
Hiroshima University \\
Higashi-Hiroshima 739-8526, Japan}
\email{inoue100@hiroshima-u.ac.jp}
\address{Graduate School of Innovation Management \\
Tokyo Institute of Technology \\
2-12-1 Ookayama 152-8552, Tokyo, Japan \\
and PRESTO, Japan Science and Technology Agency \\
4-1-8 Honcho Kawaguchi, Saitama 332-0012, Japan}
\email{nakano@craft.titech.ac.jp}

\date{June 8, 2009}


\begin{abstract}
We present an approach to the dynamic valuation of 
exposure risks in the multi-period setting, which 
incorporates a dynamic and multiple diversification of risks 
in Pareto optimal sense. 
This approach extends classical indifference premium principles and 
can be applied for the valuation of insurance risks. 
In particular, our method produces explicit computation formulas 
for the dynamic version of 
the exponential premium principles. 
Moreover, we show limit theorems asserting that the risk 
loading for our valuation decreases to zero when the number 
of divisions of a risk goes to infinity. 
\end{abstract}

\maketitle


\section{Introduction}\label{sec:1}

In premium calculations, the insurer generally requires a premium rule to 
charge a conservative margin, the so-called risk loading or 
safety loading, in exchange for accepting the insurance risk. 
When dealing with a portfolio of insurance contracts, 
the resulting risk comes from the uncertain time and size of 
loss in individual insurance contracts. 
Thus, the risk loading should reflect both the multi-period 
and characteristic risks of individual products. 
In other words, the premium should be determined via an appropriate 
valuation of the cash-flows generated by several contracts. 
For this problem, W\"{u}thrich et al.\ \cite{WBF} develop a 
multidimensional valuation method using the state price deflator. 
Dynamic market methods are exploited by Delbaen and Haezendonck \cite{DH}, 
M{\o}ller and Steffensen \cite{MS}, and the references therein. 

In this paper, we present a valuation method for 
portfolios of cash flows, which incorporates a dynamic and multiple 
diversification of risks in Pareto optimal sense. 
Consider a portfolio of $n$ liabilities 
$Z=\sum_{(i,s)\in\T_{n,t}}C_{i,s}$, 
where $\T_{n,t}=\{1,\dots,n\}\times\{t,\dots,T\}$, 
and $(C_{i,s})$ is adapted to a given filtration $(\mathcal{F}_s)$. 
Then we call an adapted process $(X_{i,s})_{(i,s)\in\T_{n,t}}$ 
with suitable integrability conditions
a {\em diversification} or {\em allocation} of $Z$ if 
\begin{equation*}
 \sum_{(i,s)\in\T_{n,t}}X_{i,s}=Z
\end{equation*}
(see Definition \ref{def:2.1} below for the precise statement). 
We consider a multidimensional (or matrix-valued) conditional expected utility 
$(E[u_{i,s}(\cdot)|\mathcal{F}_t])_{(i,s)\in\T_{n,t}}$ 
to evaluate the cash flow $(X_{i,s})_{(i,s)\in\T_{n,t}}$, and 
we define the utility $U_t=U_{n,t}$ of $Z$ by the following dynamic 
version of sup-convolution of $E[u_{i,s}(\cdot)|\mathcal{F}_t]$'s: 
\begin{equation*}
 U_{t}(Z)=\esssup_{(X_{i,s})}\sum_{(i,s)\in\T_{n,t}}
  E[u_{i,s}(X_{i,s})|\mathcal{F}_t], 
\end{equation*}
where the essential supremum is taken over all diversifications $(X_{i,s})$ of $Z$. 
The term {\em convolution} comes from convex analysis (cf. Rockafellar 
\cite{Ro}). In the context of mathematical finance, 
the convolutions of static risk measures or monetary utility functions 
are discussed in Delbaen \cite{De}, Barrieu and El Karoui \cite{BE},  
Jouini et al.\ \cite{JST}, and Kl\"oppel and Schweizer \cite{KS}. 
The advantage of using the convolution is that 
a maximizer $(X_{i,s})$ becomes a Pareto optimal allocation of 
$Z$ (see Proposition \ref{prop:2.3} below). 
Thus, the utility $U_{t}$ induces a reconstruction of the 
cash flow $Z=\sum_{(i,s)}C_{i,s}$ in Pareto optimal sense, based on 
the insurer's risk preference 
$(E[u_{i,s}(\cdot)|\mathcal{F}_t])_{(i,s)\in\T_{n,t}}$. 

In the insurance literature, Pareto optimality has been considered by 
Arrow \cite{Ar}, Borch \cite{Bo1,Bo2}, B\"uhlmann \cite{Bu2}, 
Gerber \cite{G}, and many others.  
Recently, many authors study the allocation problems with 
risk measures. 
See Acciao \cite{A}, \cite{BE}, Burgert and R\"uschendorf \cite{BR}, 
Heath and Ku \cite{HK}, and \cite{JST}. 
Usually, Pareto optimality is discussed in terms of 
several economic agents such as in reinsurance 
and equilibrium theory. 
We employ another view; we consider 
Pareto optimality in evaluating portfolios of cash flows for a single 
agent. 

The first aim of this paper is to study the following premium principle 
$H_{t}=H_{n,t}$ defined by the indifference principle for $U_{t}$: 
\begin{equation*}
 H_{t}(Z)=\essinf\{K : U_{t}(K-Z)\ge U_{t}(0)
  \;\text{a.s.}\}, 
\end{equation*}
where $K$'s are taken from $\mathcal{F}_t$-measurable random variables
(see Definition \ref{def:2.5}). 
This premium principle generalizes the so-called principle of 
zero-utility as stated in, e.g., B\"uhlmann \cite{Bu1}. 
A financial counterpart of this valuation method is called  
the {\it indifference pricing} method, which has been widely used 
methods in incomplete markets (see, e.g., 
Hodges and Neuberger \cite{HN}, Rouge and El Karoui \cite{RE}, 
Musiela and Zariphopoulou \cite{MZ1,MZ2}, Bielecki et al.\ \cite{BJR}, and 
M{\o}ller and Steffensen \cite{MS}). 
The premium $H_{t}(Z)$ is the minimum capital requirement for which 
the insurer with utility $U_t$ is willing to sell the risk $Z$. 
If there exists a maximizer for the essential supremum of 
$U_{t}(H_{t}(Z)-Z)$, then, by Proposition \ref{prop:2.3} below, 
the risk $H_{t}(Z)- Z$ allows for a Pareto optimal allocation. 
In other words, we can interpret $H_{t}(Z)$ as the minimal amount 
such that the resulting residual risk 
$H_{t}(Z)-Z$ is diversified by a Pareto optimal allocation 
and is preferable to zero risk with respect to the preference defined by 
$U_{t}$. 

The second aim of this paper is to study the asymptotic behavior of 
$H_{n,0}$ as the number of divisions of a risk goes to infinity. 
We show two such limit theorems. 
The first one states that the indifference premium $H_{n,0}(Z)$ 
converges to $E(Z)$ as $n$ goes to infinity. 
The second one concerns the large number of divisions of time. 
These results may give a different view to the general principle of insurance 
systems, which is usually explained by the law of large numbers 
for IID random variables.

This paper is organized as follows. 
In Section \ref{sec:2.1}, 
we give rigorous definitions of 
$U_{t}$ and $H_{t}$, and exhibit some basic properties of them. 
In Section \ref{sec:2.2}, we study the maximization problem defined by $U_t$.  
Section \ref{sec:2.3} is devoted to the case of 
exponential utility functions. 
In Section \ref{sec:3}, we study asymptotic behaviors of $H_{n,0}$. 
Finally, in Section 4, we apply our approach 
to products of fixed payment type, including 
life insurance products and bank loans.

\section{indifference premium}\label{sec:2}

\subsection{Definitions and general properties}\label{sec:2.1}

Let $n,T$ be positive integers, and consider index sets 
$\T_{m,\tau}:=\{1,\dots,m\}\times\{\tau,\dots,T\}$ 
$(m=1,\dots,n,\;\tau=0,\dots,T)$. 
Let $(\Omega,\mathcal{F},\{\mathcal{F}_t\}_{t=0,\dots,T},P)$ 
be a filtered probability space. 
We work on $L^{\infty}_t:=L^{\infty}(\Omega,\mathcal{F}_t,P)$, 
$t=0,\dots,T$, for the space of exposure risks. 
All inequalities and equalities applied to random variables are 
meant to hold $P$-a.s. 
We consider an element $Z\in L^{\infty}_T$ as the sum of cash flows of risks 
to be valued, discounted by some reference asset. 
Examples include the following life insurance contract: 
\begin{equation*}
 Z=\sum_{(i,s)\in\T_{n,1}} c_{i,s} 1_{(s-1<\tau_i\le s)}, 
\end{equation*}
where $c_{i,s}$ is the discounted payment to be paid at time 
$s=1,\dots,T$ if the $i$-th insured dies in the interval $(s-1,s]$, and 
$\tau_i$ denotes the future life time of the $i$-th insured.

We define the diversification of risk as follows: 
\begin{defn}
\label{def:2.1}
We call a process $(X_{i,s})_{(i,s)\in\T_{m,t}}$ a 
{\em diversification} or {\em allocation}
of $Z\in L^{\infty}_T$ if 
\begin{equation*}
 \sum_{(i,s)\in\T_{m,t}}X_{i,s}=Z, \quad 
 X_{i,s}\in L^{\infty}_s, \quad (i,s)\in\T_{m,t}. 
\end{equation*}
We write the set of all diversifications of $Z$ as 
$\mathcal{A}_{m,t}(Z)$. 
\end{defn}

For $(i,s)\in\T_{n,0}$, 
let $u_{i,s}:\mathbb{R}\to\mathbb{R}$ be a strictly increasing, strictly 
concave function of class $C^1$, satisfying 
\begin{equation*}
 u_{i,s}(0)=0, \;\; u_{i,s}^{\prime}(0)=1, \;\; u_{i,s}^{\prime}(+\infty)=0, 
 \;\; u_{i,s}^{\prime}(-\infty)=\infty, 
\end{equation*}
and assume that the agent's risk preference for the $i$-th risk at time 
$s\in\{0,\dots,T\}$, evaluated at $t\in\{0,\dots,s\}$, is 
described by the conditional expected utility 
$E[u_{i,s}(\cdot)|\mathcal{F}_t]$. 
Thus, the risk of the cash flow $Z=\sum_{(i,s)\in\T_{n,t}}X_{i,s}$ 
is evaluated by the conditional expected utility matrix 
$(E[u_{i,s}(X_{i,s})|\mathcal{F}_t])_{(i,s)\in\T_{n,t}}$. 

Now we introduce the utility $U_{t}$ defined by 
\begin{equation}
\label{eq:2.1}
 U_{t}(Z)=\esssup_{(X_{i,s})\in\mathcal{A}_{n,t}(Z)}
  \sum_{(i,s)\in\T_{n,t}}E[u_{i,s}(X_{i,s})|\mathcal{F}_t], 
   \quad Z\in L^{\infty}_T. 
\end{equation}
This is a version of sup-convolution in convex analysis, 
which takes into account the multiperiod information structure. 

The next proposition shows the basic relationship between 
$U_t$ and Pareto optimality. 
\begin{prop}
\label{prop:2.3}
 Suppose that $(X_{i,s})\in\mathcal{A}_{t}(Z)$ attains the 
essential supremum in $(\ref{eq:2.1})$. Then $(X_{i,s})$ is a 
Pareto optimal allocation of $Z$ in the sense that for 
$(Y_{i,s})\in\mathcal{A}_{n,t}(Z)$, 
\begin{align*}
 & E[u_{i,s}(Y_{i,s})|\mathcal{F}_t]\ge E[u_{i,s}(X_{i,s})|\mathcal{F}_t], \;\; 
  \forall (i,s)\in\I\times\T_{t} \\ 
&\Longrightarrow \;\; 
 E[u_{i,s}(Y_{i,s})|\mathcal{F}_t]\ge E[u_{i,s}(X_{i,s})|\mathcal{F}_t], \;\; 
  \forall (i,s)\in\I\times\T_{t}. 
\end{align*}
\end{prop}
\begin{proof}
 Suppose that there exist $(Y_{i,s})_{(i,s)\in\T_{n,t}}$ and 
$(k,\tau)\in\T_{n,t}$ such that
\begin{gather*}
 E[u_{i,s}(Y_{i,s})|\mathcal{F}_t]\ge E[u_{i,s}(X_{i,s})|\mathcal{F}_t], 
  \;\; \forall (i,s)\in\T_{n,t}, \\
 P(E[u_{k,\tau,t}(Y_{k,\tau})|\mathcal{F}_t]> 
   E[u_{k,\tau,t}(X_{k,\tau})|\mathcal{F}_t])>0. 
\end{gather*}
Then $\sum_{(i,s)}E[u_{i,s}(Y_{i,s})|\mathcal{F}_t]\ge 
\sum_{(i,s)}E[u_{i,s}(X_{i,s})|\mathcal{F}_t]$
and this inequality is strict with positive probability. 
However this contradicts the optimality of $(X_{i,s})$. 
\end{proof}

By the proposition above, the utility $U_{t}$ induces a 
reconstruction of the cash flow $Z=\sum_{(i,s)}C_{i,s}$ 
in Pareto optimal sense, 
based on the insurer's risk preference 
$(E[u_{i,s}(\cdot)|\mathcal{F}_t])$. 
It should be noted that Pareto optimality here means 
the non-inferiority in multi-objective optimization. 

The next proposition shows that many properties of 
$(E[u_{i,s}(\cdot)|\mathcal{F}_t])_{(i,s)\in\T_{n,t}}$ 
carry over to $U_{t}$. 
\begin{prop}
\label{prop:2.4}
The conditional utility $U_{t}$ maps $L^{\infty}_t$ to 
 $L^{\infty}_t$ with $U_t(0)=0$ and satisfies the following properties: 
\begin{enumerate}
 \item Monotonicity: $U_{t}(X)\ge U_{t}(Y)$ for  $X,Y\in L^{\infty}_T$ such that 
  $X\ge Y $. 
 \item Concavity: $U_{t}(aX+(1-a)Y)\ge aU_t(X)+(1-a)U_t(Y)$ for 
 $X,Y\in L^{\infty}_T$ and $a\in (0,1)$. 
\end{enumerate} 
\end{prop}
\begin{proof}
 It follows from $u_{i,s}(x)\le x$ $(x\in\mathbb{R})$ for $(i,s)\in\T_{n,t}$ 
 that $U_t(X)\le E[X|\mathcal{F}_t]$. In particular, $U_t$ maps to $L_t^{\infty}$ 
 and satisfies $U_t(0)\le 0$.  
 Considering the trivial diversification $0=\sum_{(i,s)}X_{i,s}$ with $X_{i,s}=0$, 
 we find $U_t(0)\ge 0$. Thus $U_t(0)=0$ follows. 
 
To prove the monotonicity, let $X\ge Y$ and $(Y_{k,s})\in\mathcal{A}_{n,t}(Y)$, 
and take $(X_{k,s})\in\mathcal{A}_{n,t}(X)$ defined by 
\begin{equation*}
 X_{k,s}=\begin{cases}
	  Y_{k,s}, & (k,s)\neq (i,T), \\
          Y_{k,s}+X-Y, & (k,s)=(i,T). 
	 \end{cases}
\end{equation*}
Then, by the monotonicity of $u_{i,T}(\cdot)$, 
\begin{equation*}
 U_{t}(X)\ge \sum_{(k,s)\in\T_{n,t}}E[u_{k,s}(X_{k,s})|\mathcal{F}_t]
  \ge \sum_{(k,s)\in\T_{n,t}}E[u_{k,s}(Y_{k,s})|\mathcal{F}_t]. 
\end{equation*}
Taking the supremum of the right-hand side over $(Y_{k,s})$, we get 
$U_{t}(X)\ge U_{t}(Y)$D

For $(X_{i,s})\in\mathcal{A}_{n,t}(X)$, 
$(Y_{i,s})\in\mathcal{A}_{n,t}(Y)$, 
and $a\in (0,1)$, we have 
$(aX_{i,s}+(1-a)Y_{i,s})\in\mathcal{A}_{n,t}(aX+(1-a)Y)$. 
From this the concavity of $U_{t}$ follows easily. 
\end{proof}

Now we shall introduce a premium principle by the indifference 
valuation with respect to the utility $U_{t}$. 
\begin{defn}
\label{def:2.5}
We call the $\mathcal{F}_t$-measurable random variable $H_{t}(Z)$ given by 
\begin{equation*}
 H_{t}(Z):=\essinf\left\{K\in L^{\infty}_t: 
  U_{t}(K-Z)\ge U_{t}(0) \right\}
\end{equation*}
the \emph{indifference premium} of $Z\in L^{\infty}_T$ at time $t=0,\dots,T$. 
\end{defn}

We exhibit some elementary properties of the indifference premium $H_{t}$. 
\begin{prop}
\label{prop:2.7}
The indifference premium $H_{t}$ maps $L^{\infty}_t$ to $L^{\infty}_t$ with 
$H_t(0)=0$ and 
satisfies the following properties: 
\begin{enumerate}
 \item Monotonicity: $H_{t}(X)\ge H_{t}(Y)$ for $X,Y\in L^{\infty}_T$ such that 
  $X\ge Y$. 
 \item Convexity: $H_t(aX+(1-a)Y)\le aH_t(X)+(1-a)H_t(Y)$ for 
  $X,Y\in L^{\infty}_T$ and $a\in (0,1)$. 
 \item Risk loading property: 
 $H_{t}(Z)\ge E(Z|\mathcal{F}_t)$, $Z\in L^{\infty}_T$. 
 \item Translation invariance: 
  $H_{t}(Z+C)=H_{t}(Z)+C$ for $Z\in L^{\infty}_T$ and $C\in L^{\infty}_t$. 
\end{enumerate}
\end{prop}
\begin{proof}
If $K\in L^{\infty}_t$ satisfies  
$U_{t}(K-Z)\ge U_{t}(0)$, then by $U_{t}(Z)\le E(Z|\mathcal{F}_t)$ 
we have $0\le E[K-Z|\mathcal{F}_t]=K-E[Z|\mathcal{F}_t]$, whence 
$H_t(Z)\ge E[Z|\mathcal{F}_t]$. 
Thus $H_t$ maps to $L_t^{\infty}$ and satisfies $H_t(0)\ge 0$. 
Since $H_t(0)\le 0$ is trivial, $H_t(0)=0$ follows. 

To see the monotonicity, let $X,Y\in L^{\infty}_T$ with $X\ge Y$. 
Then for any $K\in L^{\infty}_t$ 
 satisfying $U_{t}(K-X)\ge U_{t}(0)=0$, the monotonicity of 
$U_{t}$ gives $U_{t}(K-Y)\ge U_{t}(0)$, implying $H_{t}(X)\ge H_{t}(Y)$. 

Let $X,Y\in L^{\infty}_T$ and $a\in (0,1)$. For $K,L\in L^{\infty}_t$ 
satisfying $U_{t}(K-X)\ge 0$ and $U_t(L-Y)\ge 0$, 
we have from the concavity of $U_{t}$ that 
$U_{t}(aK+(1-a)L-aX-(1-a)Y)\ge 0$. 
Thus $H_{t}(aX+(1-a)Y)\le aK+(1-a)L.$ 
Since $K,L$ are arbitrary, the convexity of $H_{t}$ follows. 

To prove the translation invariance, let $Z\in L^{\infty}_T$ and $C\in L^{\infty}_t$. 
If $K\in L^{\infty}_t$ satisfies 
$U_{t}(K-Z-C)\ge 0$, then 
$K-C\ge H_{t}(Z)$. Thus $H_{t}(C+Z)\ge H_{t}(Z)+C$. 
On the other hand, if $K\in L^{\infty}_t$ satisfies 
$U_{t}(K-Z)\ge 0$, then 
in view of $K-Z=K+C-Z-C$, we have 
$K+C\ge H_{t}(Z+C)$, which leads to 
$H_{t}(Z)\ge H_{t}(Z+C)-C$.

\end{proof}

\begin{rem}
\label{rem:2.14}
Recall that a sequence of mappings 
$\rho_t: L^{\infty}_T\to L^{\infty}_t$, $t=0,\dots,T$, is called 
a {\em dynamic convex risk measure} if the following conditions are satisfied: 
\begin{enumerate}
 \item If $X\le Y$, then $\rho_t(X)\ge\rho_t(Y)$. 
 \item $\rho_t$ is convex. 
 \item $\rho_t(X+K)=\rho_t(X)-K$ for $X\in L^{\infty}_T$ and 
  $K\in L^{\infty}_t$. 
\end{enumerate}
See, e.g., F\"ollmer and Penner \cite{FP} and 
Frittelli and Rosazza Gianin \cite{FR}. 
Therefore, the mappings $\rho_t: L^{\infty}_T\to L^{\infty}_t$, $t=0,\dots,T$, 
defined by 
\begin{equation*}
 \rho_t(Z):=H_{t}(-Z), \quad Z\in L^{\infty}_T 
\end{equation*}
give a dynamic convex risk measure on $L^{\infty}_T$. 
This is a dynamic counterpart of the connection 
between premium principles and static risk measures.  
\end{rem}

\subsection{Optimal diversification problem}\label{sec:2.2}

We shall study the maximization problem in (\ref{eq:2.1}). 
\begin{thm}
\label{thm:2.15}
Let $(X_{i,s})\in\mathcal{A}_{n,t}$. Then 
$(X_{i,s})$ is the maximizer for $(\ref{eq:2.1})$ if and only if 
$(u_{i,s}^{\prime}(X_{i,s}))_{s=t}^T$ 
does not depend on $i=1,\dots,n$ and 
 $(u_{1,s}^{\prime}(X_{1,s}))_{s=t}^T$ is a martingale. 
\end{thm}
\begin{proof}
Suppose that $(u_{i,s}^{\prime}(X_{i,s}))_{s=t}^T$ 
does not depend on $i\in\I$ and that 
 $(u_{1,s}^{\prime}(X_{1,s}))_{s=t}^T$ is a martingale. 
For $(Y_{i,s})\in\mathcal{A}_{n,t}(Z)$, 
the concavity of $u_{i,s}$'s and the martingale property give 
\begin{align*}
 &\sum_{(i,s)\in\T_{n,t}}E[u_{i,s}(Y_{i,s})|\mathcal{F}_t] 
 - \sum_{(i,s)\in\T_{n,t}}E[u_{i,s}(X_{i,s})|\mathcal{F}_t] \\ 
 & \le \sum_{(i,s)\in\T_{n,t}}E[u^{\prime}_{i,s}(X_{i,s})
  (Y_{i,s}-X_{i,s})|\mathcal{F}_t] 
 = \sum_{(i,s)\in\T_{n,t}}E[u^{\prime}_{1,T}(X_{1,T})
    (Y_{i,s}-X_{i,s})|\mathcal{F}_t] \\
 &= E\left[u^{\prime}_{1,T}(X_{1,T})\sum_{(i,s)\in\T_{n,t}}
    (Y_{i,s}-X_{i,s})\;\bigg|\;\mathcal{F}_t\right] =0. 
\end{align*}
Thus $(X_{i,s})$ is optimal. 

Conversely, suppose that $(X_{i,s})\in\mathcal{A}_{n,t}(Z)$ 
is optimal. 
Take distinct $(k,\tau),(j,r)\in\T_{n,t}$ with $\tau\ge r$ 
and $A\in\mathcal{F}_r$. Consider for $y\in\mathbb{R}$, 
\begin{equation*}
 Y_{i,s}^y:= 
\begin{cases}
 X_{i,s}+y1_A & \text{if}\; (i,s)=(k,\tau), \\
 X_{i,s}-y1_A & \text{if}\; (i,s)=(j,r), \\
 X_{i,s} & \text{otherwise}. 
\end{cases}
\end{equation*}
Then $(Y_{i,s}^y)\in\mathcal{A}_{n,t}(Z)$. Moreover  
the optimality of $(X_{i,s})$ implies the (random) function 
\begin{equation*}
 f(y)=\sum_{(i,s)\in\T_{n,t}}
    E[u_{i,s}(Y^y_{i,s})|\mathcal{F}_t], \quad y\in\mathbb{R}, 
\end{equation*}
becomes maximal at $y=0$ almost surely. The condition $f^{\prime}(0)=0$ implies 
\begin{equation*}
 E[(u^{\prime}_{k,\tau}(X_{k,\tau})-u^{\prime}_{j,r}(X_{j,r}))1_A]=0, 
\end{equation*}
which leads to 
$u^{\prime}_{k,\tau}(X_{k,\tau})=u^{\prime}_{j,\tau}(X_{j,\tau})$ 
and to the martingale property of $(u^{\prime}_{k,s}(X_{k,s}))$. 
\end{proof}

We consider the reduction of 
the maximization problem (\ref{eq:2.1}) to the case $n=1$. 
To this end, we define the sup-convolution $u^{(n)}_s$ 
of $(u_{i,s})_{i=1,\dots,n}$ by 
\begin{equation}
\label{eq:2.9.5}
 u^{(n)}_s(x)=\sup\left\{\sum_{i=1}^n u_{i,s}(x_i) : 
 x=\sum_{i=1}^n x_i\right\}, \quad x\in\mathbb{R}. 
\end{equation}
We exhibit basic properties of $u_s^{(n)}$.  
\begin{prop}
\label{prop:2.16}
The sup-convolution 
$u^{(n)}_{s}:\mathbb{R}\to\mathbb{R}$ is also 
a strictly increasing, strictly concave function of class $C^1$, satisfying 
\begin{equation*}
 u^{(n)}_{s}(0)=0, \;\; (u_{s}^{(n)})^{\prime}(0)=1, \;\; 
 (u_{s}^{(n)})^{\prime}(+\infty)=0, 
 \;\; (u_{s}^{(n)})^{\prime}(-\infty)=\infty, \quad s=0,\dots,T. 
\end{equation*}
Moreover, if $I_{i,s}$ denotes the inverse function of 
$u_{i,s}^{\prime}$, $(i,s)\in\T_{n,t}$, then 
\begin{equation}
\label{eq:2.8} 
 u_s^{(n)}(x)=\ts\sum_{i=1}^n u_{i,s}\left(I_{i,s}
  \left(\ts\sum_{j=1}^nI_{j,s}\right)^{-1}(x)\right), \quad x\in\mathbb{R}. 
\end{equation}
\end{prop}
\begin{proof}
For $x\in\mathbb{R}$, consider 
$x_{i}:=I_{i,s}((\sum_{j=1}^n I_{j,s})^{-1}(x))$.  
Then, $x=\sum_{i=1}^n x_{i}$. 
For any $(y_i)_{i=1,\dots,n}$ with 
$\sum_{i=1}^n y_i=x$, the concavity of $u_{i,s}$'s implies 
\begin{equation*}
 \ts\sum_{i=1}^n u_{i,s}(y_i)-\ts\sum_{i=1}^n u_{i,s}(x_{i})\le 
  \ts\sum_{i=1}^n u_{i,s}^{\prime}(x_i)(y_i-x_i)
  = (\ts\sum_{j=1}^nI_{i,s})^{-1}(x)\sum_{i=1}^n (y_i-x_i)
  =0, 
\end{equation*}
whence $(x_i)$ is optimal and (\ref{eq:2.8}) holds. 
From this, the other assertions follow easily. 
\end{proof}

\begin{prop}
\label{prop:2.17}
For $Z\in L^{\infty}_T$, the utility $U_{t}(Z)$ is given by 
\begin{equation}
\label{eq:2.9}
 U_{t}(Z)=\esssup_{(X_s)\in\mathcal{A}_{1,t}(Z)}
  \sum_{s=t}^T E[u^{(n)}_s(X_s)|\mathcal{F}_t]. 
\end{equation}
\end{prop}
\begin{proof}
For $(Y_s)\in\mathcal{A}_{1,t}(Z)$, put 
$X_{i,s}^Y:=I_{i,s}((\sum_{j=1}^n I_{i,s})^{-1}(Y_s))$. 
Then by Proposition \ref{prop:2.16}
\begin{align*}
 U_{t}(Z)&\le \esssup\left\{\sum_{s=t}^T
  E[u^{(n)}_s(Y_s)|\mathcal{F}_t] : (Y_s)\in\mathcal{A}_{1,t}(Z)\right\} 
  \\
 &= \esssup\left\{\sum_{s=t}^T\sum_{i=1}^n
  E[u_{i,s}(X^Y_{i,s})|\mathcal{F}_t] : (Y_s)\in\mathcal{A}_{1,t}(Z) 
  \right\} \le U_{t}(Z). 
\end{align*}
Thus (\ref{eq:2.9}) follows. 
\end{proof}

In the rest of this section, we denote $u_s^{(n)}$ by $u_s$ for simplicity. 
By Theorem \ref{thm:2.15} and Proposition \ref{prop:2.17}, 
we have the following: 
\begin{cor}
\label{cor:2.18}
 A necessary and sufficient condition for 
$(X_{s})\in\mathcal{A}_{1,t}$ to be the maximizer for 
 $(\ref{eq:2.9})$ is that 
 $(u_{s}^{\prime}(X_{s}))_{s=t}^T$ is a martingale. 
\end{cor}

From the theorem above, the problem is now reduced to that of 
finding $(X_{s})\in\mathcal{A}_{1,t}(Z)$ 
such that $(u_{s}^{\prime}(X_{s}))_{s=t}^T$ is a martingale.
We adopt a duality approach to this problem. 
Define the function $u^{*}_s: [0,\infty)\to [0,\infty]$ by 
\begin{equation*}
 u^{*}_s(y)=\sup_{x\in\mathbb{R}}\{u_s(x)-xy\}. 
\end{equation*}
Then, for a positive martingale $(M_s)_{s=t}^T$, we see that 
\begin{equation*}
 \sum_{s=t}^T E[u_s(X_s)|\mathcal{F}_t]
  - \sum_{s=t}^T E[M_sX_s|\mathcal{F}_t]
  \le \sum_{s=t}^T E[u^{*}_s(M_s)|\mathcal{F}_t]. 
\end{equation*}
Since $(X_s)_{s=t}^T\in\mathcal{A}_{1,t}(Z)$ and 
$(M_s)$ is a martingale, 
the second term on the left-hand side
can be written as $E[M_TZ|\mathcal{F}_t]$. 
In view of this observation, 
denoting by $\mathcal{M}_t$ the set of all positive martingales 
$(M_s)_{s=t}^T$, we obtain, 
for $(X_s)\in\mathcal{A}_{1,t}(Z)$ 
and $(M_s)_{s=t}^T\in\mathcal{M}_t$, 
\begin{equation}
\label{eq:2.10}
  \sum_{s=t}^T E[u_s(X_s)|\mathcal{F}_t]
  \le \sum_{s=t}^T E[u^{*}_s(M_s)|\mathcal{F}_t] + E[M_TZ|\mathcal{F}_t]. 
\end{equation}
Thus we are led to the following dual problem: 
\begin{equation}
\label{eq:2.11}
  \essinf_{(M_s)\in\mathcal{M}_t} 
  \left\{\sum_{s=t}^T E[u^{*}_s(M_s)|\mathcal{F}_t] 
   + E[M_TZ|\mathcal{F}_t] \right\}. 
\end{equation}

\begin{thm}
\label{thm:2.19}
There exists a unique $(M_s)\in\mathcal{M}_t$ that attains the 
essential infimum in $(\ref{eq:2.11})$. 
\end{thm}
\begin{proof}
The uniqueness follows easily from the strict convexity of $u_s^{*}(y)$. 

To prove the existence, set 
\begin{equation*}
 \Psi(M)_t:=\sum_{s=t}^T E[u_s^{*}(M_s)|\mathcal{F}_t] 
  + E[M_TZ|\mathcal{F}_t], \quad M\in\mathcal{M}_t. 
\end{equation*}
The family $\{\Psi(M)_t\}_{M\in\mathcal{M}_t}$ is closed under 
pairwise minimization, i.e., 
\begin{equation*}
 M^1,M^2\in\mathcal{M}\;\;\Rightarrow\;\; 
 \min\{\Psi(M^1)_t,\Psi(M^2)_t\}\in\{\Psi(M)_t\}_{M\in\mathcal{M}_t}. 
\end{equation*}
In fact, for any $M^1,M^2\in\mathcal{M}_t$, we 
put $A=\{\Psi(M^1)_t\le\Psi(M^2)_t\}\in\mathcal{F}_t$, 
and consider $L\in\mathcal{M}_t$ defined by 
$L_s=M_s^1 1_{A}+M_s^2 1_{A^c}$. Then it is easy to 
see that $\Psi(L)_t=\Psi(M^1)_t 1_{A} + \Psi(M^2)_t 1_{A^c}
 =\min\{\Psi(M^1)_t,\Psi(M^2)_t\}$.  

Thus, from Neveu \cite[Proposition VI-1-1]{Ne}, 
there exists a sequence $(M_s^{(m)})\in\mathcal{M}$ such that 
\begin{equation*}
 \lim_{m\to\infty}
  \left\{\sum_{s=t}^T E[u^{*}_s(M_s^{(m)})|\mathcal{F}_t] 
   + E[M_T^{(m)} Z|\mathcal{F}_t] \right\}
  = \essinf_{(M_s)\in\mathcal{M}_t}\Psi(M)_t, \quad\text{a.s.},  
\end{equation*}
where the convergence is monotone nonincreasing. 
This and the monotone convergence theorem give
\begin{equation*}
 \lim_{m\to\infty}E[\Psi(M^{(m)})_t]
 =E\left[\essinf_{(M_s)\in\mathcal{M}_t}\Psi(M)_t\right]
 \le E(Z)<\infty. 
\end{equation*}
Hence we find that for $s=t,\dots,T$, 
\begin{equation*}
 \sup_{m}E[u^{*}_s(M_s^{(m)})]<+\infty. 
\end{equation*}
From this and Lemma \ref{lem:2.20} below, 
the sequence $(M_s^{(m)})_{m=1}^{\infty}$, 
$s=t,\dots,T$ is uniformly integrable 
(in particular, bounded in $L^1(\Omega,\mathcal{F}_T,P)$). 
Thus a multidimensional version of Koml\'{o}s's theorem 
(see Koml\'{o}s \cite{K} and Remark \ref{rem:2.21} below) 
implies that there exist 
a subsequence $(M_s^{(m_k)})$ and $(\tilde{M}_s)$ such that 
\begin{equation}
\label{eq:2.12}
 \tilde{M}_s=\lim_{k\to\infty}\frac{M_s^{(m_1)}+\dots +M_s^{(m_k)}}{k}, 
 \quad s=t,\dots,T, \;\;\text{a.s.}
\end{equation}

Since $u^{*}_s(\cdot)$ is convex, we have for every $s=t,\dots,T$, 
\begin{equation*}
 \sup_{k}E\left[u_s^{*}\left(\frac{M_s^{(m_1)}+\dots +M_s^{(m_k)}}{k}
  \right)\right]\le \sup_{k}E\left[
  \frac{u^{*}_s(M_s^{(m_1)})+\dots +u^{*}_s(M_s^{(m_k)})}{k}\right]
 < +\infty. 
\end{equation*}
Hence the convergence in (\ref{eq:2.12}) also occurs in $L^1(\Omega,\mathcal{F}_s,P)$, 
which implies that $(\tilde{M}_s)$ is a nonnegative martingale. 
Using Fatou's lemma, we have 
\begin{align*}
 \sum_{s=t}^TE[u^{*}_s(\tilde{M}_s)|\mathcal{F}_t] 
  + E[\tilde{M}_TZ|\mathcal{F}_t]
 &\le \lim_{k\to\infty}\frac{1}{k}\sum_{i=1}^k 
  \left\{\sum_{s=t}^TE[u^{*}_s(M_s^{(m_i)})|\mathcal{F}_t]
 + E[M_T^{(m_i)}Z|\mathcal{F}_t] \right\} \\
 &= \essinf_{(M_s)\in\mathcal{M}_t}\Psi(M)_t. 
\end{align*}
Therefore, to complete the proof, it suffices to show 
that $\tilde{M}_s$ is actually positive. To this end, 
we observe that for $L_s\equiv 1$, 
\begin{align*}
 0&\le \lim_{\varepsilon\downarrow 0}\frac{
   \Psi_t((1-\varepsilon)\tilde{M}+\varepsilon L)-\Psi_t(\tilde{M})}
    {\varepsilon} \\
  &= \lim_{\varepsilon\downarrow 0}\frac{1}{\varepsilon}
   \sum_{s=t}^T E[u_s^*((1-\varepsilon)\tilde{M}_s+\varepsilon)
    - u_s^*(\tilde{M}_s)|\mathcal{F}_t] 
   + E[Z(1-\tilde{M}_T)|\mathcal{F}_t] \\ 
  &=\sum_{s=t}^T E[-I_s(\tilde{M}_s)(1-\tilde{M}_s)|\mathcal{F}_t]
  + E[Z(1-\tilde{M}_s)|\mathcal{F}_t],   
\end{align*}
where $I_s=(u_s^{\prime})^{-1}$. 
This together with 
\begin{equation*}
 0\le u_s^{*}(\tilde{M}_s)= u_s^{*}(\tilde{M}_s)-u_s^*(1)
 \le -(u_s^{*})^{\prime}(M_s)(1-M_s) 
 =I_s(\tilde{M}_s)(1-\tilde{M}_s)
\end{equation*}
gives
\begin{equation*}
 0\le E[I_s(\tilde{M}_s)(1-\tilde{M}_s)1_{\{\tilde{M}_s=0\}}]
 \le E[I_s(\tilde{M}_s)(1-\tilde{M}_s)]< \infty. 
\end{equation*}
However, since $I_s(0)=+\infty$, we have $\tilde{M}_s>0$. 
\end{proof}

In the proof above, we have used the following lemma: 
\begin{lem}
\label{lem:2.20}
It holds that 
\begin{equation*}
 \lim_{y\to\infty}\frac{u^{*}_s(y)}{y}= +\infty, \quad s=0,\dots,T. 
\end{equation*}
\end{lem}
\begin{proof}
By Proposition \ref{prop:2.17}, we have $I_s(0+)=+\infty$, 
$I_s(+\infty)=-\infty$, and 
\begin{equation*}
 u_s^{*}(y)=u_s(I_s(y))-yI_s(y), \quad y>0. 
\end{equation*}
Moreover 
\begin{equation*}
 (u_s^{*})^{\prime}(y)=-I_s(y), \quad y>0. 
\end{equation*}
Since $u_s^{*}$ is convex and $I_s(+\infty)=-\infty$, we find that 
$\lim_{y\to\infty}u_s^{*}(y)=+\infty$. 
Thus the lemma follows from de l'Hospital's theorem. 
\end{proof}

\begin{rem}
\label{rem:2.21} 
It is straightforward to extend Koml\'{o}s's theorem to a 
 multidimensional one. Indeed, applying \cite[Theorem 1]{K} for $s=t$, 
we take a subsequence $\{n_k^{t}\}\subset\{1,2,\dots\}$. 
Next applying \cite[Theorem 1]{K} for $s=t+1$ and $\{n_k^{t+1}\}$
we again choose a subsequence $\{n_k^{t+1}\}\subset\{n_k^{t}\}$. 
Repeating this procedure, we obtain a subsequence $\{n_k^{T}\}$ 
which, by \cite[Theorem 1a]{K}, satisfies the desired convergence property. 
\end{rem}

\begin{thm}
\label{thm:2.22}
Let $(X_s)_{s=t}^T$ and $(M_s)_{s=t}^T$ satisfy 
$X_s=I_s(M_s)$, $s=t,\dots,T$. 
Then for $Z\in L^{\infty}_T$, the following conditions are equivalent: 
\begin{enumerate}
 \item $(X_s)$ is in $\mathcal{A}_{1,t}(Z)$ 
  and attains the essential supremum in $(\ref{eq:2.9})$; 
 \item $(M_s)$ belongs to $\mathcal{M}_t$ with 
  $I_s(M_s)\in L^{\infty}_s$, $s=t,\dots,T$, and 
  is the minimizer for the problem $(\ref{eq:2.11})$. 
\end{enumerate}
Moreover, if one of $(i)$ and $(ii)$ holds, then 
\begin{equation*}
 U_{t}(Z)= 
 \essinf_{(M_s)\in\mathcal{M}_t} 
  \left\{\sum_{s=t}^T E[u^{*}_s(M_s)|\mathcal{F}_t] 
   + E[M_TZ|\mathcal{F}_t] \right\}.
\end{equation*}
\end{thm}
\begin{proof}
Suppose that $(X_s)_{s=t}^T$ satisfies (i). 
For $(L_s)\in\mathcal{M}_t$, the convexity of $u_s^{*}$ gives 
\begin{align*}
 &\sum_{s=t}^T E[u^{*}_s(L_s)|\mathcal{F}_t] + E[L_TZ|\mathcal{F}_t]
 - \sum_{s=t}^T E[u^{*}_s(M_s)|\mathcal{F}_t] - E[M_TZ|\mathcal{F}_t] \\
 &\ge \sum_{s=t}^T E[(u^{*}_s)^{\prime}(M_s)
     (L_s-M_s)|\mathcal{F}_t] + E[(L_T-M_T)Z|\mathcal{F}_t]. 
\end{align*}
Since $(u_s^{*})^{\prime}(M_s)=-I_s(M_s)$, $(M_s)\in\mathcal{M}_t$ 
and $\sum_{s=t}^T I_s(M_s)=Z$, the right-hand side in the 
above inequality is equal to 
\begin{align*}
 & \sum_{s=t}^T E[-I_s(M_s)E[L_T-M_T | \mathcal{F}_s]|\mathcal{F}_t] 
    + E[(L_T-M_T)Z|\mathcal{F}_t] \\ 
 &= E\left[\left(Z-\sum_{s=t}^T I_s(M_s)\right)(L_T-M_T)
    \: \bigg|\: \mathcal{F}_t\right]
  = 0, 
\end{align*}
whence $(M_s)$ is a solution. 

Conversely, suppose $(M_s)$ satisfies (ii). 
Then, there exists $K>0$ such that $I_s(M_s)\le K$. Since 
$I_s(\cdot)$ is decreasing, we have $M_s\ge u^{\prime}_s(K)$. 
Thus, for some $\varepsilon>0$, 
\begin{equation*}
 M_s\ge\varepsilon, \quad s=t,\dots,T. 
\end{equation*}
Now, fix $A\in\mathcal{F}$ and define $(L_s^y)$ by 
\begin{equation*}
 L_s^y:=M_s + y P(A | \mathcal{F}_s), \quad y> -\varepsilon. 
\end{equation*}
Since $(L_s^y)\in\mathcal{M}_t$, the function 
\begin{equation*}
 f(y):=\sum_{s=t}^T E[u^{*}_s(L^y_s)|\mathcal{F}_t]
  + E[L_T^y Z|\mathcal{F}_t], \quad y> -\varepsilon, 
\end{equation*}
becomes minimal at $y=0$. Hence from $f^{\prime}(0)=0$, 
\begin{equation*}
 0=\sum_{s=t}^T E[-I_s(M_s)E[1_A | \mathcal{F}_s]] 
    + E[ZE[1_A | \mathcal{F}_T]] 
  = E\left(Z-\sum_{s=t}^T I_s(M_s)\right)1_A. 
\end{equation*}
Since $A\in\mathcal{F}$ is arbitrary, we have 
\begin{equation*}
 \sum_{s=t}^T X_s=\sum_{s=t}^T I_s(M_s)=Z. 
\end{equation*}
Moreover we find that $X_s=I_s(M_s)\in L^{\infty}_s$ and 
that $u_s^{\prime}(X_s)=M_s$. 
Thus by Corollary \ref{cor:2.18}, 
$(X_s)\in\mathcal{A}_{1,t}^{\infty}(Z)$ is the optimal solution 
to the problem (\ref{eq:2.9}). 

Finally, suppose one of the conditions (i) and (ii). Then from 
$u_s^{*}(y)=u_s(I_s(y))-yI_s(y)$ and $X_s=I_s(M_s)$, 
\begin{equation*}
 \sum_{s=t}^T E[u_s(X_s)|\mathcal{F}_t] 
 = \sum_{s=t}^T E[u^{*}_s(M_s)|\mathcal{F}_t] + E[M_TZ|\mathcal{F}_t]. 
\end{equation*}
Thus the desired equality follows. 
\end{proof}


\subsection{The case of exponential utilities}\label{sec:2.3}

In this section, to investigate $H_t$ in more details, 
we consider a class of 
exponential utility functions $(u_{i,s})_{(i,s)\in\T_{n,0}}$, 
each of which defined by 
\begin{equation}
\label{eq:2.13}
 u_{i,s}(x)=\frac{1}{\alpha_{i,s}}(1-e^{-\alpha_{i,s} x}), 
  \quad x\in\mathbb{R}, \;\; (i,s)\in\T_{n,0},  
\end{equation}
where $\alpha_{i,s}\in (0,\infty)$ for all $(i,s)\in\T_{n,0}$. 
In view of Proposition \ref{prop:2.16}, using the 
sup-convolution $u_s(x)=u_s^{(n)}(x)$ defined by (\ref{eq:2.9.5}), 
we may consider the case $n=1$. 
An elementary calculation shows that $u_s$ is again an 
exponential utility function given by 
\begin{equation*}
 u_s(x)=\frac{1}{\alpha_s}(1-e^{-\alpha_s x}), 
 \quad x\in\mathbb{R}, \;\; s=0,\dots,T, 
\end{equation*}
where $\alpha_s$ is defined by 
\begin{equation*}
 \frac{1}{\alpha_s}=\sum_{i=1}^s\frac{1}{\alpha_{i,s}}, 
 \quad s=0,\dots,T. 
\end{equation*}

Now, for $Z\in L^{\infty}_T$, we define 
the adapted process $(V_t(Z))_{t=0}^T$ by the backward iteration 
\begin{equation*}
\begin{cases}
 V_t(Z)=-\ds\frac{1}{\beta_{t+1}}
  \log E\left[\left. e^{-\beta_{t+1}V_{t+1}(Z)} 
   \right|\mathcal{F}_{t}\right], & t=0,\dots,T-1, \\
 V_T(Z)= Z, 
\end{cases}
\end{equation*}
as well as the adapted process 
$(\widehat{X}_s(Z))_{s=t}^T= (\widehat{X}_s^{(t)}(Z))_{s=t}^T$ by
\begin{equation}
\label{eq:2.16.5}
\begin{cases}
 \widehat{X}_s(Z)=\ds\frac{1}{\alpha_s}\left\{
   \beta_t V_t(Z)+\sum_{r=t+1}^{s}\beta_r(V_r(Z)-V_{r-1}(Z))\right\}, 
   & s=t+1,\dots,T,\\
 \widehat{X}_t(Z)= \ds\frac{\beta_t}{\alpha_t}V_t(Z),  
\end{cases}
\end{equation}
where $(\beta_t)_{t=0}^T$ is the modified risk aversion parameter 
defined by 
\begin{equation}
\label{eq:2.14}
 \frac{1}{\beta_t}=\sum_{s=t}^T\frac{1}{\alpha_s}. 
\end{equation}
We can now completely describe the optimizer for 
(\ref{eq:2.9}). 
\begin{thm}
\label{thm:2.23}
For $Z\in L^{\infty}_T$, the process 
$(\widehat{X}_s(Z))_{s=t}^T$ is a unique maximizer for the essential 
supremum in $(\ref{eq:2.9})$. 
Moreover $U_{t}(Z)$ is given by
\begin{equation*}
 U_{t}(Z)=\frac{1}{\beta_{t}}\{1-\exp\left(-\beta_t V_t(Z)\right)\}. 
\end{equation*}
\end{thm}
\begin{proof}
From Corollary \ref{cor:2.18} and $u_s^{\prime}(x)=e^{-\alpha_s x}$, 
our task is to find a positive martingale $(M_s)$ such that 
\begin{equation}
\label{eq:2.15}
 \prod_{s=t}^T M_s^{-1/\alpha_s}=e^Z. 
\end{equation}
In fact, from $(M_s)$ we obtain a desired solution 
$X_s:=(-1/\alpha_s)\log M_s$. 
Every positive martingale $M_s$ is represented as 
$M_s=\prod_{r=t}^s \xi_r$, where $(\xi_r)$ is positive and adapted, 
and satisfies 
$E(\xi_r|\mathcal{F}_{r-1})=1$ for $r\ge t+1$. Using this representation, 
we can write the condition (\ref{eq:2.15}) as 
\begin{equation}
\label{eq:2.16}
 \prod_{r=t}^T \xi_r^{-1/\beta_r}
 =\prod_{r=t}^T \prod_{s=r}^T \xi_r^{-1/\alpha_s}
 =\prod_{s=t}^T\prod_{r=t}^s \xi_r^{-1/\alpha_s} = e^Z. 
\end{equation}
However, if we put 
\begin{equation*}
 \xi_t:=e^{-\beta_t V_t}, \quad 
 \xi_s:=\exp(-\beta_s(V_s(Z)-V_{s-1}(Z))), \quad s=t+1,\dots,T, 
\end{equation*}
then we see that $(\xi_s)$ satisfies 
$E(\xi_s|\mathcal{F}_{s-1})=1$, $s\ge t+1$, and (\ref{eq:2.16}). 
Hence, $\widehat{M}_s:=\prod_{r=t}^s\xi_s$ satisfies (\ref{eq:2.15}).  
Now we find that $\widehat{X}_s$ in (\ref{eq:2.16.5}) is written as 
\begin{equation*}
 \alpha_s\widehat{X}_s= - \sum_{r=t}^s\beta_r\log \xi_s 
 = - \log\prod_{r=t}^s \xi_s = -\log\widehat{M}_s. 
\end{equation*}
Thus $\widehat{X}_s$ is optimal. 
The uniqueness follows from $\widehat{X}_s=I_s(\widehat{M}_s)$ and 
Theorems \ref{thm:2.19} and \ref{thm:2.22}. 
Finally, we have
\begin{equation*}
 U_{t}(Z)=\sum_{s=t}^T\frac{1}{\alpha_s}-\sum_{s=t}^T
  \frac{1}{\alpha_s}E(e^{-\alpha_s\widehat{X}_s}|\mathcal{F}_t) 
 =\frac{1}{\beta_t}\left(1-e^{-\alpha_t\widehat{X}_t}\right)
 =\frac{1}{\beta_t}\left(1-e^{-\beta_tV_t(Z)}\right), 
\end{equation*}
as desired. 
\end{proof}

Next, we turn to the indifference premium $H_{t}(Z)$. 
Set 
\begin{equation*}
\mathcal{M}_t^0=\left\{(M_s)_{s=t}^T : 
 \text{positive martingale}, \; M_t=1\right\}. 
\end{equation*}
\begin{thm}
\label{thm:2.24}
The indifference premium $H_{t}(Z)$ of 
$Z\in L^{\infty}_T$ is determined by 
the backward iteration 
\begin{equation}
\label{eq:2.17}
\begin{cases}
 H_{t}(Z)=\ds\frac{1}{\beta_{t+1}}
  \log E\left[\left. e^{\beta_{t+1}H_{t+1}(Z)} 
   \right|\mathcal{F}_t\right], & t=0,\dots,T-1, \\
 H_{T}(Z)= Z,  
\end{cases}
\end{equation} 
and $H_{t}(Z)$ is represented as 
\begin{equation}
\label{eq:2.18}
  H_{t}(Z)=\esssup_{(M_s)\in\mathcal{M}_t^0} 
   \left\{E[M_TZ|\mathcal{F}_t] - 
    \sum_{s=t}^T \frac{1}{a_s} E[M_s\log M_s|\mathcal{F}_t]\right\}, 
  \;\; t=0,\dots, T.  
\end{equation}
Moreover $H_{t}$ satisfies the following dynamic programming 
property: 
\begin{equation}
\label{eq:2.19}
 H_{t}(Z)=H_{t}(H_{t+\tau}(Z)), 
 \quad t=0,\dots,T-\tau, \;\; \tau=1,\dots,T. 
\end{equation}
\end{thm}
\begin{proof}
By Theorems \ref{thm:2.22} and \ref{thm:2.23} we have 
\begin{equation*}
  U_{t}(Z)=\essinf_{K\in L^{\infty}_{t,+}} \Phi_t(K;Z) 
\end{equation*}
with 
\begin{equation*}
  \Phi_t(K;Z)=\essinf_{(M_s)\in\mathcal{M}_t^0} 
    \left\{\sum_{s=t}E[u_s^{*}(KM_s)|\mathcal{F}_t] 
	+ E[M_TKZ|\mathcal{F}_t]\right\}. 
\end{equation*}
Here we have denoted by $L^{\infty}_{t,+}$ the set of 
all $Y\in L^{\infty}_t$ with $Y\ge 0$. 

On the other hand, we can write 
\begin{equation*}
  E[u_s^{*}(KM_s)|\mathcal{F}_t] 
   = \frac{1}{\alpha_s}(1-K) + \frac{1}{\alpha_s}K\log K 
   + \frac{K}{\alpha_s} E[M_s\log M_s |\mathcal{F}_t]. 
\end{equation*}
Thus we get
\begin{equation*}
  \Phi_t(K;Z)=\sum_{s=t}^T u_s^{*}(K) + \Psi_t(1;Z). 
\end{equation*}
However, since $(u_s^{*})^{\prime}(K)=(1/\alpha_s)\log K$, 
the essential infimum of $\Phi_t(K;Z)$ is attained by 
$K=e^{-\beta_t\Phi_t(1;Z)}$, whence  
\begin{equation*}
  U_{t}(Z)=\frac{1}{\beta_t}\left(1-e^{-\beta_t\Phi_t(1;Z)}\right). 
\end{equation*}
Thus, $\Phi_t(1;Z)=V_t(Z)$ and for $K\in L^{\infty}_t$, 
$0\le U_{t}(K-Z)$ 
if and only if $0\le \Phi_t(1;K-Z)$. 
Since $\Phi_t(1;K-Z)=K+\Phi_t(1;-Z)$, 
we deduce that  $H_{t}=-\Phi_t(1;-Z)=-V_t(-Z)$. 
Hence (\ref{eq:2.17}) and (\ref{eq:2.18}) hold. 
Since $E[M_s\log M_s|\mathcal{F}_t]\ge 0$, 
\begin{equation*}
 H_{t}(0)=\esssup_{(M_s)\in\mathcal{M}_t^0} 
   \left\{-\sum_{s=t}^T \frac{1}{\alpha_s} 
   E[M_s\log M_s|\mathcal{F}_t]\right\} = 0. 
\end{equation*}
From this and the translation invariance, we deduce that 
\begin{equation}
\label{eq:2.20}
 H_{t}(K)=K, \quad K\in L^{\infty}_t. 
\end{equation}
Using this and induction, 
we finally obtain the dynamic programming property (\ref{eq:2.19}). 
The iteration formulas (\ref{eq:2.17}) and (\ref{eq:2.20}) imply that 
\begin{equation*}
 H_{t}(H_{t+1}(Z))=\frac{1}{\beta_{t+1}}
  \log E[e^{\beta_{t+1}H_{t+1}(H_{t+1}(Z))}
   |\mathcal{F}_t] = H_{t}(Z). 
\end{equation*}
Hence we have (\ref{eq:2.19}) for $\tau=1$. Suppose that 
(\ref{eq:2.19}) holds for $\tau\in\{1,\dots,T-1\}$. Then, 
\begin{align*}
 H_{t}(H_{t+\tau+1}(Z))
 &=\frac{1}{\beta_{t+1}}\log E[\exp(\beta_{t+1}H_{t+1}
    (H_{t+\tau+1}(Z)))|\mathcal{F}_t] \\
 &= \frac{1}{\beta_{t+1}}\log E[\exp(\beta_{t+1}H_{t+1}
    (Z))|\mathcal{F}_t] \\ 
 &=H_{t}(Z), \qquad t=0,\dots,T-\tau-1. 
\end{align*}
Thus (\ref{eq:2.19}) follows. 
\end{proof}

\begin{rem}
We have the following Pareto optimal allocation of $H_{t}(Z)-Z$: 
\begin{equation*}
 H_{t}(Z)-Z=\sum_{s=t}^T \widehat{X}_s. 
\end{equation*}
Since $V_r(H_{t}(Z)-Z)=-H_{r}(-H_{t}(Z)-Z)
= H_{t}(Z)-H_{r}(Z)$ for $r\ge t$, 
the allocation $\widehat{X}_s$ is given by 
\begin{equation*}
 \widehat{X}_t=0, \quad \widehat{X}_s=-\frac{1}{\alpha_s}
  \sum_{r=t+1}^s \beta_r(H_{r}(Z)-H_{r-1}(Z)), 
  \;\; s=t+1,\dots,T. 
\end{equation*}
\end{rem}

\begin{rem}
Recall that a dynamic convex risk measure $\rho_t$ is called 
{\em time-consistent} if 
\begin{equation*}
 \rho_t(Z)=\rho_t(-\rho_{t+\tau}(Z)), 
 \quad t=0,\dots,T-\tau, \;\; \tau=1,\dots,T, 
 \quad Z\in L^{\infty}_T 
\end{equation*}
(see, e.g., \cite{FP}). 
From Remark \ref{rem:2.14} and Theorem \ref{thm:2.24}, 
the sequence of mappings 
\begin{equation*}
 \rho_t(Z):=H_{t}(-Z), \quad Z\in L^{\infty}_T, 
\end{equation*}
becomes a dynamic convex risk measure on $L^{\infty}_T$ with 
time-consistency.  
\end{rem}

\begin{rem}
\label{rem:2.27}
Suppose that $\mathcal{F}_0$ is $P$-trivial and 
that $\mathcal{F}_1=\mathcal{F}_T$. Then, 
by (\ref{eq:2.17}) we have $H_{t}(Z)=Z$ $(t\ge 1)$, 
whence 
\begin{equation*}
 H_{0}(Z)=\frac{1}{\beta_1}\log E[e^{\beta_1Z}], 
\end{equation*}
which is the classical exponential premium principle. 
\end{rem}

\begin{rem}
Our dynamic diversification approach gives the 
recursive formula (\ref{eq:2.17}) slightly different 
from that for the entropic risk measures in \cite{MZ2} and \cite{FP}. 
Indeed, the dynamic convolution produces the 
modified risk aversion parameter $(\beta_t)$ that is usually time-dependent. 
In particular, if all $\alpha_s$'s are identical to some $\alpha>0$ 
then $\beta_t=\alpha/(T-t+1)$. 
\end{rem}


\section{Large diversification effect}\label{sec:3}

In this section, we shall study the asymptotics of 
$H_{n,0}$ as the number of divisions of a risk goes to infinity.  
We assume that $\mathcal{F}_0$ consists of all null sets from 
$\mathcal{F}_T$ and their compliments. 
Hence all $\mathcal{F}_0$-measurable random variables are constants a.s. 

\subsection{Diversification over a large number of products}
\label{sec:3.2.1}

Consider the functions $u_{i,s}:\mathbb{R}\to\mathbb{R}$, 
$i=1,2,\dots$, $s=0,\dots,T$, such that each $u_{i,s}$ is strictly increasing and 
strictly concave and of class $C^2$, with 
\begin{equation}
\label{eq:3.0.7}
 u_{i,s}(0)=0, \;\; u_{i,s}^{\prime}(0)=1, \;\; 
 u_{i,s}^{\prime}(\infty)=0, \;\; u_{i,s}^{\prime}(-\infty)=\infty. 
\end{equation}
Recall from Section \ref{sec:2.1} that the utility 
$U_{0}(Z)=U_{n,0}(Z)$ of $Z\in L^{\infty}_T$ 
is given by 
\begin{equation*}
 U_{n,0}(Z)=\sup_{(X_{i,s})\in\mathcal{A}_{n,0}(Z)}
  \sum_{(i,s)\in\T_{n,0}}E[u_{i,s}(X_{i,s})],  
\end{equation*}
and that the indifference premium 
$H_{0}(Z)=H_{n,0}(Z)$ satisfies the risk loading property 
\begin{equation*}
 H_{n,0}(Z)\ge E(Z). 
\end{equation*}

We have the following convergence result:
\begin{thm}
\label{thm:3.3}
Suppose that for each $i=1,2,\dots,$ and $s=0,\dots,T$ the function 
$u^{\prime\prime}_{i,s}(x)$ is nondecreasing and satisfies 
\begin{equation}
\label{eq:3.1}
 \sum_{i=1}^{\infty}\int_0^{\delta}\frac{\delta-\lambda}
  {u_{i,s}^{\prime\prime}(I_{i,s}(1+\lambda))}d\lambda=-\infty, 
 \quad\forall\delta >0, \;\; s=0,\dots,T, 
\end{equation}
where $I_{i,s}=(u_{i,s}^{\prime})^{-1}$. Then, 
\begin{equation}
\label{eq:3.2}
 \lim_{n\to\infty}H_{n,0}(Z)=E(Z). 
\end{equation}
\end{thm}
\begin{rem}
 Since $u_{i,s}^{\prime\prime}$ is nondecreasing, we have 
\begin{equation}
\label{eq:3.3}
 \frac{\delta^2}{2u_{i,s}^{\prime\prime}(0)}
  \le\int_0^{\delta}\frac{\delta-\lambda}
   {u_{i,s}^{\prime\prime}(I_{i,s}(1+\lambda))}d\lambda. 
\end{equation}
Thus, the condition (\ref{eq:3.1}) is stronger than 
\begin{equation*}
 \sum_{i=1}^{\infty}\frac{1}{u_{i,s}^{\prime\prime}(0)}=-\infty, 
 \quad s=0,\dots,T. 
\end{equation*}
For example, the family of exponential utility functions
\begin{equation*}
 u_{i,s}(x)=\frac{1}{\alpha_{i,s}}(1-e^{-\alpha_{i,s} x}), 
  \quad \alpha_{i,s}>0 
\end{equation*}
satisfies (\ref{eq:3.1}) when 
$\sum_{i=1}^{\infty}(1/\alpha_{i,s})=+\infty$, $s=0,\dots,T$. 
\end{rem}

\begin{proof}[The proof of Theorem \ref{thm:3.3}]
By Proposition \ref{prop:2.17}, 
for any $(Y_s)\in\mathcal{A}_{n,0}(Z)$, 
\begin{equation}
\label{eq:3.5}
 U_{n,0}(Z)\ge \sum_{s=0}^T E[u^{(n)}_s(Y_s)]. 
\end{equation}
By Lemma \ref{lem:3.5} below and the monotone convergence theorem, 
\begin{equation*}
 \lim_{n\to\infty}E[u^{(n)}_s(Y_s)]=E[Y_s], \quad s=0,\dots,T. 
\end{equation*}
This and (\ref{eq:3.5}) give 
\begin{equation*}
 \lim_{n\to\infty}U_{n,0}(Z)=E(Z). 
\end{equation*}

We notice that the function $x\mapsto U_{n,0}(x-Z)$ is 
increasing and concave, whence continuous, on $\mathbb{R}$. Thus we have 
\begin{equation*}
 U_{n,0}(H_{n,0}(Z)-Z)=0, \quad n=1,2,\dots.
\end{equation*}
On the other hand, from Proposition \ref{prop:2.7} 
$H_{n,0}(X)\ge E(Z)$. 
Also, since $U_{n,0}(Z)$ is nondecreasing in $n$, 
\begin{equation*}
 U_{n+1,0}(H_{n,0}(Z)-Z)
 \ge U_{n,0}(H_{n,0}(Z)-Z)= 0. 
\end{equation*}
From this and the definition of $H_{n,0}$, $H_{n+1,0}(Z)\le H_{n,0}(Z)$. 
Hence putting 
$H_{\infty,0}(Z):=\lim_{n\to\infty}H^{(n)}_0(Z)$, 
we have 
\begin{equation*}
 E[H_{\infty,0}(Z)-Z]=\lim_{n\to\infty}
  U_{n,0}(H_{\infty,0}(Z)-Z)\le 
  \lim_{n\to\infty}U_{n,0}(H_{n,0}(Z)-Z)=0. 
\end{equation*}
Thus $H_{\infty,0}(Z)=E(Z)$. 
\end{proof}

In the proof of the above theorem, we have used the following lemma: 
\begin{lem}
\label{lem:3.5}
Under the assumption of Theorem \ref{thm:3.3}, we have for each $s=0,\dots,T$, 
\begin{equation*}
 \lim_{n\to\infty}u_s^{(n)}(x)=x, \quad\forall
  x\in\mathbb{R}. 
\end{equation*} 
\end{lem}
\begin{proof}
We fix $s=0,\dots,T$ and drop the subscript $s$ on all functions 
for brevity. 
 Since $u^{(n)}(x)$ is increasing in $n$ and 
$u^{(n)}(x)\le x$ for all $x$, there exists a limit 
$u^{(\infty)}(x)=\lim_{n\to\infty}u^{(n)}(x)$ and this function 
$u^{(\infty)}$ is proper and concaveD

Let $v_i(y)$ be the conjugate of $-u_i(-x)$, i.e., 
\begin{equation*}
 v_i(y):=(-u_i(-\cdot))^{*}(y)=\sup_{x\in\mathbb{R}}(xy+u_i(-x)), 
  \quad y\in\mathbb{R}. 
\end{equation*}
Also, let $v^{(n)}(y)$ be the conjugate of $-u^{(n)}(-x)$. 
Then for each $i$, an elementary analysis of $u_i$ shows that 
\begin{equation*}
v_i(y)=
\begin{cases}
 u_i(I_i(y))-yI_i(y) & \text{if}\; y>0, \\
 +\infty & \text{if}\; y\le 0, 
\end{cases} 
\end{equation*}
and that $v_i(1)=0$. Also, 
\begin{equation*}
 v_i^{\prime}(y)=-I_i(y), \quad 
 v_i^{\prime\prime}(y)=\frac{-1}{u_i^{\prime\prime}(I_i(y))}. 
\end{equation*}
Using integration by parts, we get for each $y>0$, 
\begin{equation*}
 v_i(y)=v_i(1)+v_i^{\prime}(1)(y-1)
  +\int_1^y v_i^{\prime\prime}(t)(y-t)dt
  = \int_1^y v_i^{\prime\prime}(t)(y-t)dt.  
\end{equation*}
Hence, if $y>1$, then 
\begin{equation*}
 v_i(y)=\int_0^{y-1}v_i^{\prime\prime}(1+s)(y-1-s)ds
  =-\int_0^{y-1}\frac{y-1-s}{u_i^{\prime\prime}(I_i(1+s))}ds. 
\end{equation*}
Thus (\ref{eq:3.1}) implies 
\begin{equation}
\label{eq:3.4}
 \sum_{i=1}^n v_i(y)\to +\infty\quad (n\to\infty). 
\end{equation}
If $0<y<1$, then $v_i^{\prime\prime}$ is decreasing, so that we have
\begin{equation*}
 v_i(y)\ge v_i^{\prime\prime}(1)\int_y^1(t-y)dt
 = -\frac{(1-y)^2}{2}\cdot\frac{1}{u_i^{\prime\prime}(0)}. 
\end{equation*}
Thus, from (\ref{eq:3.3}), (\ref{eq:3.4}) again holds. 

Put $v^{(\infty)}(y):=(-u^{(\infty)}(-\cdot))^{*}(y)$. Then 
$v^{(\infty)}(y)\ge 0$ and $u^{(\infty)}(-x)\le -x$. 
Hence $v^{(\infty)}(1)=0$. 
Since $u^{(n)}(x)$ is increasing in $n$, it follows from 
\cite[Theorem 16.4]{Ro} that for $y\neq 1$, 
\begin{align*}
 v^{(\infty)}(y)&=\sup_{x\in\mathbb{R}}(yx+u^{(\infty)}(-x))
 \ge \sup_{x\in\mathbb{R}}(yx+u^{(n)}(-x)) \\
 &= v^{(n)}(y)= \sum_{i=1}^n v_i(y)\to +\infty, \quad n\to\infty. 
\end{align*}
Thus $v^{(\infty)}(y)=+\infty$ for $y\neq 1$. 
From this and \cite[Theorem 12.2]{Ro}, 
\begin{equation*}
 \liminf_{z\to x}(-u^{(\infty)}(-z))=(v^{(\infty)})^{*}(x)
 =\sup_{y\in\mathbb{R}}(yx-v^{(\infty)}(y))=x, \quad x\in\mathbb{R}. 
\end{equation*}
Thus, if we let $z\searrow x$, then from the monotonicity of 
$u^{(\infty)}$, we see that 
\begin{equation*}
 x= \liminf_{z\searrow x}(-u^{(\infty)}(-z))\ge -u^{(\infty)}(-x)\ge x,  
\end{equation*}
which implies $u^{(\infty)}(x)=x$, as desired. 
\end{proof}

In the theory of premium calculations, it is well-known that 
the exponential principle approximates the variance principle 
in the following way: 
\begin{equation*}
 \frac{1}{\alpha}\log E[e^{\alpha Z}]
 =E(Z)+\frac{\alpha}{2}\Var(Z)+O(\alpha^2), \qquad \alpha\searrow 0,  
\end{equation*}
where $O(\cdot)$ denotes Landau's symbol. 

In our dynamic setting, we have the following analogous result: 
\begin{thm}
\label{thm:3.6}
Let $H_{n,0}(Z)$ be 
the indifference premium of $Z\in L^{\infty}_T$ 
with the exponential utilities 
$(u_{i,s})$ defined by $(\ref{eq:2.13})$. 
Suppose that the sequence $(\alpha_{i,s})$ in $(\ref{eq:2.13})$ 
is bounded. Then we have
\begin{equation*}
 H_{n,0}(Z)=E(Z)+\frac{1}{2}\sum_{t=1}^T\beta^{(n)}_t 
  E[(\Delta Z_t)^2] + O(n^{-2}), \qquad n\to\infty, 
\end{equation*}
where $\beta_t^{(n)}=\beta_t$ is defined by $(\ref{eq:2.14})$. 
\end{thm}
\begin{proof}
First, using (\ref{eq:2.17}) repeatedly, we get 
\begin{equation*}
 \|H_{t}(Z)\|_{\infty}\le\|Z\|_{\infty}, \quad t=0,\dots,T-1, 
\end{equation*}
where $\|\cdot\|_{\infty}$ stands for the norm of the Banach space 
$L^{\infty}_T$. 
From this as well as $\beta_t=O(n^{-1})$ and Taylor's theorem for the function 
$x\mapsto \log E[e^{xZ}|\mathcal{F}_t]$, we find that
\begin{equation}
\label{eq:3.6}
 H_{t}(Z)=E(H_{t}(Z)|\mathcal{F}_{t-1})
  +\frac{\beta_t}{2}\Var(H_{t}(Z)|\mathcal{F}_{t-1}) 
  +R_{t}(n)n^{-2},  
\end{equation}
where $R_{t}(n)$, $n=1,2,\dots,$ are $\mathcal{F}_{t}$-measurable 
random variables satisfying $\sup_n\|R_{t}(n)\|_{\infty}$ is finite. 
In what follows, we also write $R_t(n)$ for random variables having the 
same properties, which may not be necessarily 
equal to each other. 
In particular, we find that 
\begin{equation*}
 H_{T-1}(Z)=E(Z|\mathcal{F}_{T-1}) + \frac{\beta_T}{2}
  E[(\Delta Z_T)^2|\mathcal{F}_{T-1}] + R_{T-1}(n)n^{-2}. 
\end{equation*}

Now suppose that for some $t=1,\dots,T-1$, 
\begin{equation}
\label{eq:3.7}
 H_{t}(Z)=E(Z|\mathcal{F}_{t}) + \sum_{s=t+1}^T 
  \frac{\beta_s}{2}E[(\Delta Z_s)^2|\mathcal{F}_{t}] + R_{t}(n)n^{-2}. 
\end{equation}
Then, 
\begin{equation}
\label{eq:3.8}
 E[H_{t}(Z)|\mathcal{F}_{t-1}]= E[Z|\mathcal{F}_{t-1}] 
 +\sum_{s=t+1}^T\frac{\beta_s}{2}E[(\Delta Z_s)^2|\mathcal{F}_{t-1}]
 +R_{t-1}(n)n^{-2}. 
\end{equation}
Also, using $\beta_s=O(n^{-1})$ and 
\begin{equation*}
 E\left[\Delta Z_t\left\{E[(\Delta Z_s)^2|\mathcal{F}_t]
     -E[(\Delta Z_s)^2|\mathcal{F}_{t-1}]\right\}
     \big|\mathcal{F}_{t-1}\right] = 0, 
\end{equation*}
we obtain 
\begin{align}
 &\Var(H_{t}(Z)|\mathcal{F}_{t-1}) 
   \label{eq:3.9} \\
 &=E\left[\left\{\Delta Z_t +\sum_{s=t+1}\frac{\beta_s}{2}\left(
    E[(\Delta Z_s)^2|\mathcal{F}_t]-E[(\Delta Z_s)^2|\mathcal{F}_{t-1}]
    \right)+R_t(n)n^{-2}\right\}^2\Bigg|\mathcal{F}_{t-1}\right] 
  \nonumber \\
 &= E[(\Delta Z_t)^2|\mathcal{F}_{t-1}] + R_{t-1}(n)n^{-2}. 
  \nonumber
\end{align}
Putting (\ref{eq:3.8}) and (\ref{eq:3.9}) into (\ref{eq:3.6}), 
we have (\ref{eq:3.7}) for $t-1$. 
Therefore by the mathematical induction, 
(\ref{eq:3.7}) holds for all $t=0,\dots,T-1$. 
In particular, the case $t=0$ gives the desired result. 
\end{proof}

\subsection{Diversification over a large number of time divisions}
\label{sec:3.2.2}

Here, we discuss the asymptotics of $H_{n,0}(Z)$ 
when the number of divisions of time increase to infinity. 
To this end, take a filtration $(\mathcal{F}_t)_{t\in [0,T]}$ 
with continuous parameter. 
We assume that the probability space $(\Omega,\mathcal{F},P)$ is 
complete and 
$\mathcal{F}_t=\cap_{s>t}\mathcal{F}_s$, $t\in [0,T]$, i.e., 
the filtered probability space 
$(\Omega,\mathcal{F},(\mathcal{F}_t)_{t\in [0,T]},P)$ satisfies 
the usual conditions. 

We consider the modified index set 
\begin{equation*}
\label{eq:3.0.5}
 \T^{(m)}_{n,k}:=\left\{\left(i,\frac{jT}{m}\right) : 
  i=1,\dots,n, \; j=k,\dots,m\right\}.  
\end{equation*}
Let $u_{i,s}:\mathbb{R}\to\mathbb{R}$, $i=1,\dots,n$, $s\in [0,T]$, 
be a strictly increasing, strictly 
concave function of class $C^1$, satisfying (\ref{eq:3.0.7}). 
As in Section \ref{sec:2.1}, we define the utility map 
$U_{n,0}^{(m)}(Z)=U_{n,0}(Z)$ of $Z\in L^{\infty}_T$ by 
\begin{equation*}
 U_{n,0}^{(m)}(Z):=\sup_{(X_{i,s})\in\mathcal{A}_{n,0}(Z)}
  \sum_{(i,s)\in\T^{(m)}_{n,0}}E[u_{i,s}(X_{i,s})],  
\end{equation*}
and consider the resulting indifference premium $H_{n,0}^{(m)}(Z)$. 

We have the following convergence result: 
\begin{thm}
Suppose that $u^{\prime\prime}_{i,s}(x)$ is nondecreasing and satisfies, 
for each $i=1,\dots,n$, 
\begin{equation*}
 \lim_{m\to\infty}\sum_{j=0}^{m}\int_0^{\delta}\frac{\delta-\lambda}
  {u_{i,jT/m}^{\prime\prime}(I_{i,jT/m}(1+\lambda))}d\lambda=-\infty, 
 \quad\forall\delta >0. 
\end{equation*}
Then, 
\begin{equation*}
 \lim_{m\to\infty} H_{n,0}^{(m)}(Z)=E(Z). 
\end{equation*}
\end{thm}
\begin{proof}
 Fix $u\in (0,T)$ and set $Z_u=E[Z|\mathcal{F}_u]$. 
By the concavity of $U_{n,0}^{(m)}$, we have 
\begin{equation*}
 U_{n,0}^{(m)}(Z)\ge \frac{1}{2}U_{n,0}^{(m)}(2Z_u)
  + \frac{1}{2}U_{n,0}^{(m)}(2(Z-Z_u)). 
\end{equation*}
Using the index set 
$\mathcal{T}_u^{(m)}:=\{kT/m: k=\lfloor mu/T\rfloor +1,\dots, T\}$, 
we get 
\begin{align}
\label{eq:3.10}
  U_{n,0}^{(m)}(2Z_u)&\ge\tilde{U}_{n,0}^{(m)}(2Z_u) \\
  &:=\sup\left\{\sum_{i=1}^n\sum_{t\in\mathcal{T}_u^{(m)}} 
     E[u_{i,t}(Y_{i,t})]: 
   2Z_u=\sum_{i=1}^n\sum_{t\in\mathcal{T}_u^{(m)}} Y_{i,t}, \; 
   Y_{i,t}\in L^{\infty}_u\right\} \nonumber \\ 
  &=\sup\left\{\sum_{i=1}^n E[u^{(m)}_i(Y_i)]: 
   2Z_u=\sum_{i=1}^m Y_i, \; 
   Y_i\in L^{\infty}_u\right\}, \nonumber
\end{align}
where the function $u_i^{(m)}$ is the sup-convolution defined by 
\begin{equation*}
 u_i^{(m)}(x)=\sup\left\{\sum\nolimits_{t\in\mathcal{T}_u^{(m)}}u_{i,t}(x): 
  \sum\nolimits_{t\in\mathcal{T}_u^{(m)}}x_t=x\right\}. 
\end{equation*}
In fact, for each $Y_{i,t}\in L^{\infty}_u$, 
$(i,t)\in\{1,\dots,n\}\times\mathcal{T}_u^{(m)}$, satisfying 
$2Z_u=\sum_{(i,t)\in\mathcal{T}_u^{(m)}}Y_{i,t}$, 
we define $X_{i,t}=Y_{i,t}$ if 
$(i,t)\in\{1,\dots,n\}\times\mathcal{T}_u^{(m)}$, $=0$ otherwise. 
Then $(X_t)\in\mathcal{A}_{n,0}(2Z_u)$, and 
\begin{equation*}
 \sum_{i=1}^n\sum_{t\in\mathcal{T}_u^{(m)}} 
  E[u_{i,t}(Y_{i,t})]= \sum_{(i,t)\in\T_{n,0}^{(m)}} E[u_{i,t}(X_{i,t})]
 \le U_{n,0}^{(m)}(2Z_u). 
\end{equation*}
Taking the supremum, we obtain the inequality in (\ref{eq:3.10}). 
Moreover, as in the proof of Proposition \ref{prop:2.17}, 
we get the second equality in (\ref{eq:3.10}). 

Since the number of elements of $\mathcal{T}_u^{(m)}$ goes to infinity 
as $m\to\infty$, in a way similar to Lemma \ref{lem:3.5}, we see that 
\begin{equation*}
 \lim_{m\to\infty}u^{(m)}_i(x)=x, \quad x\in\mathbb{R}, 
  \;\; i=1,\dots,n, 
\end{equation*}
whence $\lim_{m\to\infty}\tilde{U}_{n,0}^{(m)}(2Z_u)=E(2Z_u)=2E(Z)$. 
On the other hand, the right-continuous version of the martingale 
$Z_u$ converges to $Z$ a.s. as $u\to T$, whence by 
the dominated convergence theorem, 
$\lim_{u\to T}E[u_{n,T}(2(Z-Z_u))]=0$. 
This and $U_{n,0}^{(m)}(2(Z-Z_u))\ge E[u_{n,T}(2(Z-Z_u))]$ yield 
$\lim_{m\to\infty}U_{n,0}^{(m)}(Z)=E(Z)$. 
We can now complete the proof in the same way as the proof of 
Theorem \ref{thm:3.3}. 
\end{proof}


\section{Application to fixed payment insurance}\label{sec:4}

In this section, we apply the approach above 
to products of fixed payment type, including 
life insurance products and bank loans.


We consider a portfolio of $n$ contracts with duration $T$ in which 
the insurer pays a fixed payment to each insured at time $t=1,\dots,T$ 
if a specified event occurs in the interval $(t-1,t]$. 

We denote by $\tau_i$ the random time at which the $i$-th specified event 
occurs, and assume that $\tau_i$'s are mutually independent random 
variables on $(\Omega,\mathcal{F},P)$ satisfying $P(\tau_i>0)=1$ and 
$P(\tau_i>t)>0$ for all $t\in (0,\infty)$, $i=1,\dots,n$. 
Suppose that the reference asset is given by a riskless bond 
with deterministic interest rates. Then the discounted 
risk $Z$ of the portfolio of the contracts is represented as 
\begin{equation*}
 Z=\sum_{(i,t)\in\T_{n,1}}c_{i,t}1_{(t-1<\tau_i\le t)}, 
\end{equation*}
where $c_{i,t}$'s are the deterministic discounted payments.

We assume that the filtration $(\mathcal{F}_t)_{t=0,\dots,T}$ is 
given by 
\begin{equation*}
\mathcal{F}_t=\vee_{i=1}^n \sigma(\{\tau_i\le s\}: s=0,\dots,t),
 \qquad t=0,\dots,T.
\end{equation*}
For $t=0,\dots,T-1$ and $i=1,\dots,n$, 
we define the conditional probabilities $q_{i,t}$ and $p_{i,t}$ by  
\begin{equation*}
 q_{i,t}:=P(\tau_i\le t+1 | \tau_i>t), \quad 
 p_{i,t}:=1-q_{i,t}=P(\tau_i> t+1 | \tau_i>t).
\end{equation*}
They will play a basic role in the computations below. 
Notice that the following equalities hold: 
\begin{equation*}
q_{i,t} + p_{i,t} = 1, \quad t=0,\dots,T-1, \quad 
q_{i,0}=P(\tau_i\le 1), \quad p_{i,0}=P(1<\tau_i).
\end{equation*}

We need the following lemma: 
\begin{lem}
\label{lem:4.1}
Let $I$ be a nonempty subset of $\{1,\dots,n\}$, and let $Y_i$, $i\in I$, 
be integrable and $\sigma(\tau_i)$-measurable random variables. 
Then we have 
\begin{equation*}
 E\left[\prod_{i\in I}Y_i1_{(\tau_i>s)}\bigg| \mathcal{F}_s\right]
 =\prod_{i\in I}\frac{E[Y_i1_{(\tau_i>s)}]}{P(\tau_i>s)}1_{(\tau_i>s)}, 
 \quad s=0,\dots,T. 
\end{equation*}
\end{lem}
\begin{proof}
We consider the family of events 
\begin{equation*}
 \mathcal{G}_s=\left\{\cap_{j\in J}\{\tau_j>s_j\}: s_j=0,\dots,t, 
 \; j\in J,\; J\subset\{1,\dots,n\}\right\}. 
\end{equation*}
Then $\mathcal{G}_s$ contains $\Omega$, is closed under 
intersection, and generates $\mathcal{F}_s$. 
Thus in view of Williams \cite[p.\ 231]{W}, it is enough to show that 
\begin{equation*}
 E\left[\prod_{i\in I}Y_i1_{(\tau_i>s)}1_A\right]
 =E\left[\prod_{i\in I}k_i1_{(\tau_i>s)}1_A\right], 
 \quad A\in\mathcal{G}_{s}, 
\end{equation*}
where $k_i=E[Y_i1_{(\tau_i>s)}]/P(\tau_i>s)$. 
If $A$ is of the form $\cup_{j\in J}(\tau_j>s_j)$, then 
denoting $B=\cup_{j\in J\setminus (I\cap J)}(\tau_j>s_j)$, 
we see that 
\begin{align*}
 &E\left[\prod_{i\in I}Y_i1_{(\tau_i>s)}1_A\right]
 =E\left[\prod_{i\in I}Y_i1_{(\tau_i>s)}1_B\right]
 =\prod_{i\in I}E[Y_i1_{(\tau_i>s)}1_B] \\
 &=\prod_{i\in I}k_iP(\tau_i>s)P(B)
 =E\left[\prod_{i\in I}k_i1_{(\tau_i>s)}1_B\right]
 =E\left[\prod_{i\in I}k_i1_{(\tau_i>s)}1_A\right], 
\end{align*}
where we have used the fact that $\tau_i$'s are mutually 
independent and $s_j\le s$. 
Thus the lemma follows. 
\end{proof}


Now, recall that 
the indifference premium $H_{t}(Z)$, based on the expected 
exponential utilities, is given by (\ref{eq:2.17}). 

Let us introduce the sequence $(h_{i,t})_{t=1}^T$ defined by 
the following backward iteration: 
\begin{equation*}
\begin{cases}
 h_{i,T}=1, \\
 h_{i,t}=\left[e^{\beta_t c_{i,t}}q_{i,t-1} 
 + h_{i,t+1}^{\beta_t}p_{i,t-1}\right]^{1/\beta_t}, \quad t=1,\dots,T-1. 
\end{cases}
\end{equation*}
\begin{thm}
For $t=0,\dots,T$, the indifference premium $H_{t}(Z)$ has 
the following representation in terms of $(h_{i,s})$: 
\begin{equation}
\label{eq:4.2}
 H_{t}(Z)=\sum_{i=1}^n\left\{\sum_{s=1}^t c_{i,s}1_{(s-1<\tau\le s)} 
  + 1_{(t<\tau_i)}\log h_{i,t+1}\right\}, 
\end{equation}
where $\sum_{s=1}^0=0$. 
\end{thm}
\begin{proof}
 We prove (\ref{eq:4.2}) by the backward induction. For $t=T$, the equality 
(\ref{eq:4.2}) clearly holds. 
Suppose that (\ref{eq:4.2}) holds for some $t\le T-1$, and 
we write $y_{i,s}=\log h_{i,s}$.  
Then 
\begin{align*}
 H_{n,t-1}(Z)
 &= \sum_{\substack{1\le i\le n\\[1pt] 1\le s\le t-1}}c_{i,s}
     1_{\{s-1<\tau\le s\}} 
   +\frac{1}{\beta_t}\log \Theta_{t-1}
\end{align*}
with 
\begin{align*}
 \Theta_{t-1}&=E\left[\exp\left\{\beta_t\sum_{i=1}^n\left(c_{i,t}
  1_{(t-1<\tau_i\le t)}+y_{i,t+1}1_{(t<\tau_i)}\right)\right\}
    \bigg| \mathcal{F}_{t-1}\right] \\
 &= E\left[\prod_{i=1}^n\left\{\exp\left(\beta_tc_{i,t}
  1_{(t-1<\tau_i\le t)}+y_{i,t+1}1_{(t<\tau_i)}\right)1_{(t-1<\tau_i)}
  +1_{(\tau_i\le t-1)}\right\}\bigg| \mathcal{F}_{t-1}\right]. 
\end{align*}
Using Lemma \ref{lem:4.1} and the general fact that 
$\prod_{i=1}^n(a_i+b_i)=\sum_{m=0}^n\sum_{\Lambda\in\mathcal{I}_m}
\prod_{i\in\Lambda}a_i\prod_{j\notin\Lambda}b_j$ with $\mathcal{I}_m$ 
being the family of subsets of $\{1,\dots,n\}$ consisting of 
$m$ elements, we obtain  
\begin{align*}
 \Theta_{t-1}&=\sum_{m=0}^n\sum_{\Lambda\in\mathcal{I}_m}
 \prod_{j\notin\Lambda}1_{(\tau_j\le t-1)}
 E\left[\prod_{i\in\Lambda}\exp\left(\beta_tc_{i,t}1_{(t-1<\tau_i\le t)}
  + y_{i,t+1}1_{(t<\tau_i)}\right)1_{(t-1<\tau_i)}
 \bigg|\mathcal{F}_{t-1}\right] \\
 &= \sum_{m=0}^n\sum_{\Lambda\in\mathcal{I}_m}
 \prod_{j\notin\Lambda}1_{(\tau_j\le t-1)}\prod_{i\in\Lambda}
 \left(e^{\beta_tc_{i,t}}q_{i,t-1}+e^{\beta_ty_{i,t+1}}p_{i,t-1}\right)
  1_{(t-1<\tau_i)} \\
 &=\prod_{i=1}^n\left\{\left(e^{\beta_tc_{i,t}}q_{i,t-1}
   +e^{\beta_ty_{i,t+1}}p_{i,t-1}\right)1_{(t-1<\tau_i)} 
   +1_{(\tau_i\le t-1)}\right\}. 
\end{align*}
Thus $(1/\beta_t)\log\Theta_{t-1}=\sum_{i=1}^ny_{i,t}1_{(t<\tau_i)}$, 
which completes the proof. 
\end{proof}


\section{Conclusion}\label{sec:5}

In this paper, we propose a premium calculation principle determined by 
an efficient risk diversification for portfolios of cash flows. 
In so doing, we use the dynamic version of the sup-convolution of 
utility functionals to consider the effect of a Pareto optimal diversification 
of risks based on the insurer's multidimensional risk preference. 
This approach aims to give a possible theoretical foundation for 
the problem of determining the risk loading for portfolios of cash flows.   
We find explicit computation formulas for the variance and exponential 
premium principles, which extend the classical counterparts in the one period 
setting. 
We also show limit theorems asserting that the risk loading of the premium 
decreases to zero when the number of divisions of risk goes to infinity.  

In our future research, we wish to focus on the implementation 
and various extensions of the results obtained in this paper. 
Possible future research topics include 
\begin{itemize}
 \item the identification problem of the 
  multidimensional risk preference; 
 \item the incorporation of the market interest rate into our model; 
 \item the case of more complex filtration 
  $\{\mathcal{F}_{i,t}\}_{(i,t)\in\T}$ where $\T$ is a directed index set; 
 \item the case of monetary utility functionals; 
 \item the continuous time setting. 
\end{itemize}


\begin{thebibliography}{99}

\bibitem{A}Acciaio, B.: 
Optimal risk sharing with non-monotone monetary functionals. 
Finance Stoch. 11, 267--289 (2007).

\bibitem{Ar}Arrow, K.\ J.:  
Uncertainty and the welfare of medical care. 
Amer.\ Econ.\ Rev. 53, 941--973 (1963).

\bibitem{BE}Barrieu, P., El Karoui, N.:  
Inf-convolution of risk measures and optimal risk transfer. 
Finance Stoch. 9, 269--298 (2005).

\bibitem{BJR}Bielecki, T.\ R., Jeanblanc, M, Rutkowski, M.:  
``Hedging of defaultable claims,'' in 
{\it Paris-Princeton lectures on Mathematical Finance 2003}, 
eds. R.\ Carmona, Berlin: Springer, 1--132 (2004). 

\bibitem{Bo1}Borch, K.: 
The safety loading of reinsurance premiums. 
Skandinavisk Aktuarietidskrift 43, 163--184 (1960).

\bibitem{Bo2}Borch, K.: 
Equilibrium in a reinsurance market. 
Econometrica 30, 424--444 (1962).

\bibitem{Bu1}B\"{u}hlman, H.: 
Mathematical methods in risk theory. 
Springer, Berlin (1970). 

\bibitem{Bu2}B\"uhlmann, H.:  
An economic premium principle. 
ASTIN Bulletin 11, 52--60 (1980).

\bibitem{BR}Burgert, C., R\"uschendorf, L.: 
On the optimal risk allocation problem. 
Statist.\ Decisions 24, 153--171 (2006).

\bibitem{De}Delbaen, F.:
Coherent measures of risk on general probability spaces. 
In: Sandmann, K., Sch\"onbucher, P.\ J.\ (eds.): 
Advances in Finance and Stochastics, Essays in Honor of 
Dieter Sondermann, pp.\ 1--37. Springer, Berlin (2002).

\bibitem{DH}Delbaen, F., Haezendonck, J.: 
A martingale approach to premium calculation principles 
in an arbitrage-free market. 
Insurance Math.\ Econom. 8, 269--277 (1989). 

\bibitem{FP}F\"ollmer, H., Penner, I.: 
Convex risk measures and the dynamics of their penalty functions, 
Statist.\ Decisions 24, 61-96 (2006). 

\bibitem{FR}Frittelli, M., Rosazza Gianin, E.: 
Dynamic convex risk measures. 
In: Szeg\"{o}, G.\ (ed.). Risk measures for the 21st century. 
pp.\ 227-248. John Wiley \& Sons, New York (2004). 

\bibitem{G}Gerber, H.\ U.: 
Pareto-optimal risk exchanges and related decision problems. 
Astin Bulletin 10, 25--33 (1978).

\bibitem{HK}Heath, D., Ku, H.: 
Pareto equilibria with coherent measures of risk. 
Math.\ Finance 14, 163--172 (2004).

\bibitem{HN}Hodges, S.\ D., Neuberger, A.: 
Optimal replication of contingent claim under transaction costs. 
Review of Future Markets, 8, 222--239 (1989). 

\bibitem{JST}Jouini, E., Schachermayer, W., Touzi, N.:  
Optimal risk sharing for law invariant monetary utility functions. 
Math.\ Finance, 18, 269--292 (2008). 

\bibitem{KS}Kl\"oppel, S., Schweizer, M.: 
Dynamic indifference valuation via convex risk measures. 
Math.\ Finance, 17, 599--627.  

\bibitem{K}Koml\'{o}s, J.: 
A generalization of a problem of Steinhaus, 
Acta Math.\ Acad.\ Sci.\ Hung. 18, 217--229 (1967). 

\bibitem{Me}Meyer, M.: 
Continuous stochastic calculus with applications to finance. 
Chapman \& Hall, London (2000). 

\bibitem{MS}M{\o}ller, T., Steffensen, M.: 
Market-Valuation Methods in Life and Pension Insurance. 
Cambridge University Press, Cambridge (2007). 

\bibitem{MZ1}Musiela, M., Zariphopoulou, T.: 
An example of indifference prices under exponential preferences. 
Finance and Stochastics, 8, 229--239 (2004). 

\bibitem{MZ2}Musiela, M., Zariphopoulou, T.: 
A valuation algorithm for indifference prices in incomplete markets. 
Finance and Stochastics, 8, 399--414 (2004). 

\bibitem{Ne}Neveu, J.: 
Discrete-parameter martingales. 
North-Holland, Amsterdam (1975). 

\bibitem{Ro}Rockafellar, R.\ T.: 
Convex analysis. 
Princeton University Press, Princeton (1970). 

\bibitem{RE}Rouge, R., El Karoui, N. : 
Pricing via utility maximization and entropy. 
Mathematical Finance, 10, 259--276 (2000). 

\bibitem{W}Williams, D.: 
Probability with martingales. 
Cambridge University Press, Cambridge (1991). 

\bibitem{WBF}W\"{u}thrich, M.\ V., B\"{u}hlmann, H., Furrer, H.: 
Market-consistent actuarial valuation. 
Springer-Verlag, Berlin (2008). 

\end{thebibliography}
\end{document}